\documentclass[12pt, reqno]{amsart}
\usepackage[dvipsnames]{xcolor}
\usepackage{mathtools}
\mathtoolsset{showonlyrefs}
\usepackage{etoolbox}
\usepackage{blindtext}
\usepackage{latexsym}
\usepackage[pagewise]{lineno}
\usepackage[utf8]{inputenc}
\usepackage{exscale}
\usepackage{amsmath}
\usepackage{amssymb}
\usepackage{amsfonts}
\usepackage{mathrsfs}
\usepackage{amsbsy}
\usepackage{relsize}

\definecolor{mypink1}{rgb}{0.858, 0.188, 0.478}
\definecolor{mypink2}{RGB}{219, 48, 122}
\definecolor{mypink3}{cmyk}{0, 0.7808, 0.4429, 0.1412}
\definecolor{mygray}{gray}{0.6}
\definecolor{venetianred}{rgb}{0.78, 0.03, 0.08}
\definecolor{sapphire}{rgb}{0.03, 0.15, 0.4}
\definecolor{utahcrimson}{rgb}{0.83, 0.0, 0.25}
\definecolor{trueblue}{rgb}{0.0, 0.45, 0.81}
\definecolor{carminered}{rgb}{1.0, 0.0, 0.22}
\definecolor{cobalt}{rgb}{0.0, 0.28, 0.67}
\definecolor{cornflowerblue}{rgb}{0.39, 0.58, 0.93}

\usepackage[colorlinks, linkcolor=cobalt, citecolor=blue, urlcolor=blue, hypertexnames=false]{hyperref}

\usepackage{url} 
\usepackage{setspace}
\usepackage{mathtools}
\usepackage[inline]{enumitem}  
\usepackage[margin=1in]{geometry}
\usepackage{comment}
\usepackage{mathabx}
\usepackage{lmodern}
\usepackage{relsize}
\usepackage[T1]{fontenc}
\usepackage{cite}
\numberwithin{equation}{section}
\newtheorem{theorem}{Theorem}[section]
\newtheorem{Definition}{Definition}[section]
\newtheorem{Remark}{Remark}[section]
\newtheorem{proposition}{Proposition}[section]
\newtheorem{Lemma}{Lemma}[section]
\newtheorem{Corollary}{Corollary}[section]
\newcommand{\R}{\mathbb R}

\newcommand{\N}{\mathbb N} 
\usepackage{color}

\begin{document}
\title[The fractional Hermite  heat equation ]{Heat  equations associated to harmonic oscillator with exponential nonlinearity}
\author[D. G.  Bhimani]{Divyang G. Bhimani}
\address[D. G.  Bhimani]{Department of Mathematics, Indian Institute of Science Education and Research, Dr. Homi Bhabha Road, Pune 411008, India}
\email{\textcolor{blue}{\it divyang.bhimani@iiserpune.ac.in}}
\author[M.  Majdoub]{Mohamed Majdoub}
\address[M. Majdoub]{Department of Mathematics, College of Science, Imam Abdulrahman Bin Faisal University, P. O. Box 1982, Dammam, Saudi Arabia \& Basic and Applied Scientific Research Center, Imam Abdulrahman Bin Faisal University, P.O. Box 1982, 31441, Dammam, Saudi Arabia.}

\email{\textcolor{blue}{\it mmajdoub@iau.edu.sa}}
\email{\textcolor{blue}{\it med.majdoub@gmail.com}}

\author[R. Manna]{Ramesh Manna}
\address[R. Manna]{School of Mathematical Sciences, National Institute of Science Education and Research, Bhubaneswar, An OCC of Homi Bhabha National Institute, Jatni 752050, India.}
\email{\textcolor{blue}{\it rameshmanna@niser.ac.in}}

\subjclass[2020]{35K05, 35K55, 35A01, 42B35}
\keywords{Nonlinear parabolic equations, harmonic potential, exponential nonlinearity, local existence, global existence, nonexistence.}

 \begin{abstract} We investigate the Cauchy problem for a heat equation involving a fractional harmonic oscillator and an exponential nonlinearity:
 $$
 \partial_tu + (-\Delta +\varrho|x|^2)^{\beta}u=f(u), \quad (x,t)\in \R^d\times (0,\infty),
 $$
 where $\varrho\geq 0,~\beta>0$ and $f:\R \to \R$ exhibits exponential growth at infinity, with $f(0)=0$.
 We establish local well-posedness within the appropriate Orlicz spaces. Through the examination of small initial data in suitable Orlicz spaces, we obtain the existence of global weak-mild solutions. Additionally, precise decay estimates are presented for large time, indicating that the decay rate is influenced by the nonlinearity's behavior near the origin. Moreover, we highlight that the existence of local nonnegative classical solutions is no longer guaranteed when certain nonnegative initial data is considered within the appropriate Orlicz space.
  \end{abstract}

\maketitle
\section{Introduction}
\label{id}
 
This paper is concerned with the heat equation associated to fractional harmonic potential with  exponential type nonlinearity:
\begin{equation}
\label{hexp}
\left\{
\begin{array}{rcl} \partial_tu + (-\Delta +\varrho|x|^2)^{\beta}u&=&f(u), \quad (x,t)\in \R^d\times (0,\infty), \\
u(x,0)&=&u_0(x),
\end{array}
\right. 
\end{equation}
where $\varrho\geq 0,~\beta>0$ and $f:\R \to \R$ exhibit an exponential growth at infinity with $f(0)=0$. Note that the cases $\varrho = 0$ and $\beta=1$ in \eqref{hexp} correspond to the standard nonlinear heat (NLH) equation. It is worth mentioning that there has been a large amount of research on the NLH equation, and the monographs \cite{Henry, Hu, Pao, QS} cover a very comprehensive overview of the most established results on the subject. See also \cite{weissler1980local, brezis1996nonlinear, haraux1982non, Mohamed2021, majdoub2018ade, Ioku2018, majdoub2018picm, fino2020cpaa} and the references therein. It is worth highlighting that, to the best of our knowledge, \cite{RT} stands as the pioneering paper where discussions on heat equations with exponential nonlinearities were initiated.

 Recently, Bhimani et al. established in \cite[Theorem 1.2]{bhimani2022heat} a well-posedness result for \eqref{hexp} in Lebesgue spaces, specifically when the nonlinearity takes the form $f(u)=u|u|^{\gamma-1}$ (a polynomial type nonlinearity). In addition,further investigations have been conducted in this area, as evidenced by \cite[Theorems 1.2 and 1.4]{bhimani2021Adv} and \cite[Theorem 1.1]{cordero2021PsDO}. In this paper, we extend the study by focusing on the Cauchy problem for \eqref{hexp} with exponential nonlinearities. Such nonlinearities play a crucial role in various physical models concerning self-trapped beams in plasma, as highlighted  in \cite{Lam77}. {It is worth noting that significant research has been conducted on exponential nonlinearities, particularly for the $2D$ energy-critical NLS equation \cite{CIMM, IMMN} and the $2D$ energy-critical NLW equation \cite{CPAM2006, APDE2011, Duke2009}. Additionally, in \cite{NO-NLS}, the Cauchy problem for the NLS equation with exponential nonlinearity was thoroughly analyzed. The authors demonstrated the existence and uniqueness of global $H^{d/2}$-solutions for small initial data.}

As highlighted in \cite{Mohamed2021, majdoub2018ade, majdoub2018picm}, the well-posedness of \eqref{hexp} depends critically on two main factors: the choice of the initial data space and the behavior of the nonlinearity \( f \).

To conduct our analysis, we impose specific assumptions on the nonlinearity \( f \).
We assume that $f(0) = 0$ and that $f$ satisfies either condition \eqref{nflw} or condition \eqref{nfgw} below 
\begin{equation}\label{nflw}
 |f(u)-f(v)| \leq C  |u-v| \left( e^{\lambda |u|^p} +e^{\lambda |v|^p} \right),
\end{equation}
\begin{equation} \label{nfgw}
  |f(u)-f(v)| \leq C |u-v| \left( |u|^{m-1}e^{\lambda |u|^p} + |v|^{m-1} e^{\lambda |v|^p}  \right),
\end{equation}
where $C, \lambda>0, p>1$, and $m\geq 1+\frac{2p\beta}{d}$. The parameter $m$  captures the behavior of the nonlinearity $f(u)$ near the origin.
It is worth mentioning that the assumptions \eqref{nflw}, \eqref{nfgw} regarding the nonlinearity encompasses various cases such as $f(u)=\pm ue^{|u|^p}$ for \eqref{nflw},  and   $f(u)=\pm u|u|^{m-1},  e^{u}-1-u,  \pm u|u|^{m-1}e^{|u|^{q}} (q \leq p),  e^{|u|^q}-1 (q \leq p)$ for \eqref{nfgw}.

Unless otherwise specified, throughout this article, we assume that $\varrho=1$. The spectral decomposition of the Hermite operator $ H=H^{1}=-\Delta+|x|^2$ on $ \mathbb R^d $ is given by 
$$ 
H = \sum_{k=0}^\infty (2k+d) P_k, 
$$
where $ P_k $ stands for the orthogonal projection of $ L^2(\mathbb R^d) $ onto the eigenspace corresponding to the eigenvalue $(2k+d)$.  We define  the heat  propagator  associated to the fractional harmonic oscillator   $H^{\beta}$  by 
\[e^{-t H^{\beta}}u_0(x)=   \sum_{k=0}^\infty e^{-t(2k+d)^{\beta}} P_ku_0(x). \]

Although Lebesgue spaces are suitable for analyzing heat equations with power nonlinearities, our motivation in this study is to explore the behavior of exponential nonlinearities. To accommodate such nonlinearities effectively, we are inclined to consider initial data in Orlicz spaces. This choice allows us to address the specific challenges posed by exponential terms and obtain meaningful results in our analysis. The authors of \cite{AIHP2016} establish a characterization of nonlinearities \( f \) that guarantee the existence of a local bounded solution in \( L^q \) (for \( 1 \leq q < \infty \)) to the standard NLH equation for all nonnegative initial data \( u_0 \in L^q \). Notably, they establish that this condition is satisfied if and only if
\begin{eqnarray*}
\limsup_{s\to\infty}\left(\frac{f(s)}{s^{1+2q/d}}\right)<\infty\quad&\mbox{if}&\quad 1< q< \infty,\\
\int_1^\infty\,\sup_{1\leq t\leq s}\left(\frac{f(t)}{t}\right)\,\,\frac{ds}{s^{1+2/d}}< \infty \quad&\mbox{if}&\quad q=1.
\end{eqnarray*}

The Orlicz space $\exp L^p(\R^d),\, 1\leq p<\infty$  is defined as follows
\begin{equation*}
\exp L^p(\R^d)= \left\{ u \in L^1_{loc} (\R^d): \int_{\R^d} \left( e^{\frac{|u(x)|^p}{\lambda^p}} -1\right) dx < \infty  \ \text{for some} \ \lambda>0 \right\},
\end{equation*}
endowed with the Luxemburg  norm
\begin{align*} \label{Luxnorm}
\| u\|_{\exp L^p}= \inf \left\{\lambda>0:  \int_{\R^d} \left( e^{\frac{|u(x)|^p}{\lambda^p}} -1\right) dx \leq 1 \right\}. 
\end{align*}

To investigate the local well-posedness issue, we will utilize the following subspace of $\exp L^p$:
 \[\exp L^p_0(\R^d)= \left\{ u \in L^1_{loc} (\R^d): \int_{\R^d} \left( e^{\alpha |u(x)|^p} -1\right) dx < \infty  \ \text{for every} \ \alpha>0 \right\}.\] 
We say that $u$ is a \textbf{\textit{mild }}solution for the Cauchy problem \eqref{hexp} with $u_0\in \exp L_0^p(\R^d)$ if  $u \in C ([0, T ], \exp L^p_0 (\R^d ))$ {satisfies}
\begin{equation}\label{df}
u(t)=e^{-tH^{\beta}} u_0+\int_0^t e^{-(t-s)H^{\beta}} \, f(u(s)) \, ds.
\end{equation}

\begin{theorem}[Local well-posedness] \label{T1.1}
Let $u_0 \in  \exp L_0^p(\R^d)$ and $0< \beta \leq 1.$  Assume that $f$ satisfies \eqref{nflw}. Then there exists $T= T(u_0)>0$  and a unique mild solution  $u \in C([0, T], \exp L^p_0(\R^d))$ to \eqref{hexp}.
\end{theorem}
\begin{Remark}
~{\rm \begin{itemize}
    \item[($i$)] The restriction on $\beta$ in Theorem \ref{T1.1} is imposed as a consequence of Theorem \ref{dc}, specifically condition \eqref{P2}. This restriction plays a crucial role in the proof of Lemma \ref{L4.1}.
    \item[($ii$)] We would like to emphasize the significant role played by the density $C_0^\infty(\mathbb{R}^d)$ within $\exp L^p_0(\mathbb{R}^d)$ in the proof of Theorem \ref{T1.1}. 
\item[($iii$)] Theorem \ref{Thne} below asserts the non existence of local solutions within the space $\exp L^p(\mathbb{R}^d)$.
\end{itemize}}
\end{Remark}

We would like to highlight the following observation: the operator $e^{-tH^{\beta}}$ exhibits continuity at $t=0$ in the space $\exp L_0^p(\mathbb{R}^d)$ but not in $\exp L^p(\mathbb{R}^d)$. This distinction is evident from Proposition \ref{conexp0} below. Consequently, to study the equation \eqref{hexp} in the Orlicz space $\exp L^p(\mathbb{R}^d)$, we employ the concept of weak-mild solutions. Specifically, we define a weak-mild solution for the Cauchy problem \eqref{hexp} with $u_0\in \exp L_0^p(\mathbb{R}^d)$ as follows: $u$ belongs to $L^\infty((0, T), \exp L^p(\mathbb{R}^d))$ and satisfies the associated integral equation \eqref{df} in $\exp L^p(\mathbb{R}^d)$ for almost all $t\in (0, T)$, with $u(t)$ converging to $u_0$ in the weak* topology as $t\to 0$.

\begin{theorem}[Global existence] \label{Th1} Let $1< p \leq \frac{d(m-1)}{2\beta}$ and  $0< \beta \leq 1.$  Assume that $f$ satisfies \eqref{nfgw} for $m\geq p$.
Then there exists $\epsilon>0$ such that for any initial data $u_0 \in \exp L^p (\R^d)$ with $\|u_0\|_{\exp L^p} \leq \epsilon,$ there exists a global weak-mild solution 
\[ u \in L^{\infty} \left( (0, \infty), \exp L^p (\R^d) \right) \] to \eqref{hexp} satisfying 
\[ \lim_{t\to 0} \| u(t)- e^{-tH^{\beta}} u_0\|_{\exp L^p}=0.\]
Moreover we have  
\begin{align}\label{E1.4}
    \|u(t)\|_{L^a} \leq C t^{-\left(\frac 1{m-1}-\frac d{2\beta a}\right)},~t>0,
\end{align}
where $a$ satisfies 
\begin{enumerate}
    \item If $\frac d{2\beta}>\frac p{p-1}$, then $\frac d{2\beta} (m-1)<a<\frac d{2\beta} (m-1) \frac{1}{(2-m)_+}.$

    \item If $\frac d{2\beta}=\frac p{p-1}$, then $\frac d{2\beta} (m-1)<a<\frac d{2\beta} (m-1) \frac{1}{(2-m)_+}.$

    \item If $\frac d{2\beta}<\frac p{p-1}$ and $(2-m)_+<\frac{d(p-1)}{2 \beta p}$, then $\frac p{p-1}(m-1)<a<\frac d{2\beta} (m-1) \frac{1}{(2-m)_+},$
\end{enumerate}
with {$(z)_{+}$  stands for the positive part of a real number $z$.}
\end{theorem}

{
\begin{Remark}~{\rm In view of hypothesis and conclusions stated in the above theorem, some comments arise, we enumerate them in what follows.
\begin{itemize}
    \item[($i$)] The restriction on $\beta$ in Theorem \ref{Th1} is a direct consequence of condition \eqref{P2} in Theorem \ref{dc} below. 
    \item[($ii$)] The assumption $p>1$ is essential in the subsequent Corollaries \ref{C2.3} and \ref{C3.2}. 
\item[($iii$)] A natural question arises as to whether the range of \( a \) required for \eqref{E1.4} is optimal. It is likely that this range can be improved, and further investigation could refine our understanding of the conditions under which \eqref{E1.4} holds.
\item[($iv$)] Similar results have been obtained in previous studies such as \cite{fino2020cpaa, majdoub2018picm, Mohamed2021} for the standard NLH equation, where $\varrho=0$ in \eqref{hexp}.
\end{itemize}}
\end{Remark}
}

\begin{Definition} [$\exp L^p$ -classical solution]
Let $u_0 \in \exp L^p(\R^d)$ and $T >0$. A function $u \in C((0,T]; \exp L^p(\R^d)) \bigcap L^{\infty}_{\rm{loc}}(0,T; L^{\infty}(\R^d))$ is said to be $\exp L^p$-classical solution of \eqref{hexp} if $u \in C^{1,2} ((0, T ) \times \R^d)$, satisfies \eqref{hexp} in the classical sense and $u(t) \to u_0$ in the weak$^{\star}$ topology as $t \to 0$.
\end{Definition}

\begin{theorem}[Nonexistence] \label{Thne}
Assume that the nonlinear term $f$ is continuous, $f(x)\geq 0$ if $x\geq 0,$ and 
\begin{equation}
    \label{liminf}
    \liminf_{\eta\to \infty}\left(f(\eta) \, e^{-\lambda \eta^p}\right)>0,
\end{equation}
for some $\lambda>0$ and $p>1$. Then, there exists  $0\leq\,u_0 \in \exp L^p(\R^d)$ such that for every $T>0$ the Cauchy problem \eqref{hexp} with $\beta=1$ has no nonnegative $\exp L^p$- classical solution on $[0,T).$
\end{theorem}

Theorem \ref{Thne} reveals the absence of local solutions for specific data in $\exp L^p(\mathbb{R}^d)$, despite the existence of a global existence result for small data within the same space $\exp L^p(\mathbb{R}^d)$. Therefore, Theorem \ref{Thne} serves as a complement to Theorems \ref{T1.1} and \ref{Th1}, further enhancing our understanding of the behavior and limitations of solutions in the context of the considered problem. See  \cite[Theorem 2.1.3]{IRT} for a comparable result when $\varrho=0$ and $\beta=1$.
\begin{Remark}
{\rm The assumption $\beta=1$ in Theorem \ref{Thne} is a product of our chosen approach. Specifically, the proof relies on the use of the convolution formula for the Hermite heat semigroup $e^{-tH}$ as presented in \eqref{mc}. This particular choice and utilization of the convolution formula lead to the restriction $\beta=1$ in the above theorem.}
\end{Remark}
\begin{Remark}
{\rm In an upcoming study, we will provide a comprehensive characterization of the nonlinearities $f$ that allow equation \eqref{hexp} to possess a local solution within both Lebesgue spaces $L^q$ and Orlicz spaces $\exp L^q$, where $1 \leq q < \infty$. This investigation aims to elucidate the precise conditions under which local solutions exist in these function spaces, offering a deeper understanding of the problem at hand.}
\end{Remark}

We conclude the introduction with an outline of the paper. In the next section, we recall some basic facts and useful tools about Orlicz spaces. In Section \ref{S3} we give the proof of Theorem \ref{T1.1}. The fourth section is devoted to the proof of Theorem \ref{Th1}. Finally, Section \ref{S5} contains the proof of the nonexistence result given in Theorem \ref{Thne}. Along this paper, $C$ will stands for a positive constant which may have different values at different places.  
\section{Preliminaries and key estimates}
Let us begin by revisiting the definition of Orlicz spaces and summarizing some fundamental aspects. For a comprehensive understanding and further elaboration, we recommend referring to  \cite{HH2019, RR, RR1}. Additionally, we provide essential estimates that play a crucial role in our analysis.

\begin{Definition}[Orlicz space]
Let $\phi:\R^+ \to \R^+$ be a convex increasing function such that
\[\phi(0)=0=\lim_{s\to 0^+} \phi(s),~\lim_{s\to \infty} \phi(s)=\infty.\]
The Orlicz space $ L^{\phi}(\R^d)$  is defined as follows
\begin{equation*}
L^{\phi}(\R^d)= \left\{ u \in L^1_{loc} (\R^d): \int_{\R^d} \phi \left(\frac{|u(x)|}{\lambda} \right) dx < \infty  \ \text{for some} \ \lambda>0 \right\}
\end{equation*}
endowed with the Luxemburg  norm
\begin{align} \label{Luxnorm}
\| u\|_{L^{\phi}}= \inf \left\{\lambda>0:  \int_{\R^d} \phi\left( \frac{|u(x)|}{\lambda}\right) dx \leq 1 \right\}. 
\end{align}

We also consider the space 
 \[ L^{\phi}_0(\R^d)= \left\{ u \in L^1_{loc} (\R^d): \int_{\R^d} \phi\left(\frac{|u(x)|}{\lambda} \right) dx < \infty  \ \text{for every} \ \lambda>0 \right\}.\]
\end{Definition}
Ioku et al.  in  \cite[Section 2]{ioku2015MPAG} proved that
  \[ L_0^\phi(\R^d)= \overline{C_0^{\infty}(\R^d)}^{\| \cdot \|_{L^\phi}}= \text{the closure of} \  C_{0}^{\infty}(\R^d)  \  \text{in} \   L^{\phi}(\mathbb R^d). \]
Note that 
\begin{eqnarray*}
L^{\phi}(\R^d)=\begin{cases} L^{\phi}_0(\R^d)= L^p(\R^d) \quad if \   \phi(s)=s^p \ (1\leq p < \infty),\\
\exp L^p (\R^d) \quad if \   \phi(s)=e^{s^p}-1 \ (1\leq p < \infty).
\end{cases}
\end{eqnarray*}
\begin{Lemma}[Inclusion properties] \label{relaorL} \ 
\begin{enumerate}
    \item[1)] \label{pl1} (\cite[Lemma 2.3]{majdoub2018picm}) $L^q(\R^d) \cap L^{\infty}(\R^d) \hookrightarrow \exp L^p_0(\R^d) \hookrightarrow \exp L^p(\R^d)$ for  $1\leq q \leq p,$ with 
   { \begin{equation}
        \label{exp-l-p-est}
        \|u\|_{\exp L^p} \leq \frac{1}{(\log 2)^{\frac 1p}}\left( \|u\|_{L^q}+\|u\|_{{L^\infty}}\right).
    \end{equation} }
\item[2)] \label{pl2} (\cite[Lemma 2.3]{fino2020cpaa})  $L^q(\R^d) \cap L^{\infty}(\R^d) \hookrightarrow  L^{\phi}_0(\R^d) \hookrightarrow  L^{\phi}(\R^d)$ for $q\leq 2p,$ $\phi(s)=e^{s^p}-1-s^p (p>1),$ with   
\[\|u\|_{L^{\phi}(\R^d)} \leq C(p)\left( \|u\|_{L^q}+\|u\|_{{L^\infty}}\right).\] 
\item[3)] \label{pl3} (\cite[Lemma 2.4]{majdoub2018picm}) $\exp L^p(\R^d) \hookrightarrow L^q(\R^d)$ for $1\leq p \leq q<\infty,$ with
{
    \begin{equation}
        \label{L-q-exp-L-p}
        \|u\|_{L^q} \leq \left(\Gamma(\frac pq+1)\right)^{\frac 1q} \|u\|_{\exp L^p},
    \end{equation}}
    where  the Gamma function {is given by} $\Gamma(x):=\displaystyle\int_0^{\infty}s^{x-1} e^{-s} \, ds,~x>0.$
\end{enumerate}
\end{Lemma}

\begin{theorem}(\cite[Theorem 1.1]{bhimani2022heat})
\label{dc} For $1 \leq  p,\, q\leq\infty$ and $\beta >0,$ set
$
  \sigma_\beta \coloneqq  \frac{d}{2\beta} \Big|\frac{1}{p}-\frac{1}{q}\Big|.
 $
\begin{enumerate}
\item \label{P1} 
 If  $p,\, q \in (1,  \infty),$ or $(p,\, q)=(1, \infty),$ or $p=1$ and $q \in [2, \infty),$ or  $p\in (1, \infty)$ and $q=1,$ then there exists a constant $C>0$ such that
\begin{equation}\label{eq mainthm}
\|e^{-tH^\beta} g\|_{L^q} \le  \begin{cases}
C e^{-td^\beta} \|g\|_{L^p}  &  \text{if} \quad t\geq 1,\\
 C t^{-\sigma_\beta} \|g\|_{L^p} &  \text{if} \quad   0<t\leq 1.
\end{cases} 
\end{equation} 
\item \label{P2} 
If $0<\beta \leq 1,$ then the above estimate holds for all $1\leq p,\, q \leq\infty.$
\end{enumerate}
\end{theorem}

\begin{Remark} \label{ufr}~{\rm \begin{itemize}
    \item[($i$)] From \eqref{eq mainthm} we get
    \begin{equation}
    \label{q-q}
    \|e^{-tH^\beta} g\|_{L^q} \le C \|g\|_{L^q},\quad 1<q<\infty.
    \end{equation}
    \item[($ii$)] Since $t^{\sigma_\beta}e^{-td^\beta} \leq C$ for all $t\geq 1$, \eqref{eq mainthm} yields
    \begin{equation}
    \label{q-p}
     \|e^{-tH^\beta} g\|_{L^q} \le C t^{-\sigma_\beta}\|g\|_{L^p},\quad 0<t<\infty.
    \end{equation}
\end{itemize}
}
\end{Remark}

Fino-Kirane in   \cite[Proposition 1]{fino2020cpaa}  obtained  several  $L^q-\exp L^p$ estimates for the fractional heat propagator $e^{-t(-\Delta)^{\beta}}$ with $0<\beta \leq 2.$ See also \cite[Proposition 3.2]{majdoub2018picm} and \cite[Lemma 3.1]{furioli2017JDE}.  In the subsequent proposition, we extend this result to encompass the the fractional harmonic oscillator $H^{\beta}$, where $\beta>0$. More precisely, we establish the following generalization:
\begin{proposition} \label{lpexp}
Let $1<q \leq p< \infty, t>0$ and $1\leq r \leq \infty.$
Then 
\begin{enumerate}
\item \label{lpexp1} $\|e^{-tH^{\beta}} g\|_{\exp L^p} \leq C   \, \|g\|_{\exp L^{p}}$ for $\beta >0$.

\item \label{lpexp2}$\|e^{-tH^{\beta}} g\|_{\exp L^p} \leq C  t^{-\frac{d}{2\beta q}} \, \left( \log (t^{-\frac d{2\beta}} +1)\right)^{-\frac 1p} \, \|g\|_{L^q}$ for $\beta >0$.

\item \label{lpexp3} $\|e^{-tH^{\beta}} g\|_{\exp L^p} \leq \frac{{C}}{(\log 2)^{\frac 1p}}\left[ t^{-\frac{d}{2\beta r}} \|g\|_{L^r} +\|g\|_{L^q}\right]$ for $0<\beta \leq 1.$
\end{enumerate}
\end{proposition}
\begin{proof}
\eqref{lpexp1} By Taylor expansion and Theorem \ref{dc},  for $\lambda>0,$ we have 
\begin{align*}
  \int_{\R^d}\left(\exp \left|\frac{e^{-tH^{\beta}}g}{\lambda}\right|^p-1\right) \, dx &\leq  \sum_{k=1}^{\infty} \frac{{C^{pk}}\|g\|_{L^{pk}}^{pk}}{k!\lambda^{pk}} =\int_{\R^d}\left(\exp \left|\frac{{C}g}{\lambda}\right|^p-1\right) \, dx.
\end{align*}

Thus, we get
  \begin{align*}
   \|e^{-tH^{\beta}} g\|_{\exp L^p}&=\inf\left\{\lambda>0:\int_{\R^d}\left(\exp \left|\frac{e^{-tH^{\beta}}g}{\lambda}\right|^p-1\right) \, dx \leq 1\right\} \\
   &\leq \inf\left\{\lambda>0:\int_{\R^d}\left(\exp \left|\frac{{C}g}{\lambda}\right|^p-1\right) \, dx \leq 1\right\} ={C}\|g\|_{\exp L^{p}}.
  \end{align*}  
 \eqref{lpexp2} By Theorem \ref{dc} for $q\leq p$, we infer 
\begin{align*}
\int_{\R^d}\left(\exp \left|\frac{e^{-tH^{\beta}}g}{\lambda}\right|^p-1\right) \, dx
&\leq   \sum_{k=1}^{\infty} \frac{{C^{pk}}t^{-\frac{d}{2\beta}(\frac 1q-\frac{1}{pk}){pk}}\|g\|_{L^q}^{pk}}{k!\lambda^{pk}} \\
  &=t^{\frac{d}{2\beta}} \left(\exp \left(\frac{{C}t^{-\frac{d}{2\beta q}}\|f\|_{L^q}}{\lambda}\right)^p-1\right).
\end{align*}
This leads to
\[\|e^{-tH^{\beta}} f\|_{\exp L^p} \leq C  t^{-\frac{d}{2\beta q}} \, \left( \log (t^{-\frac d{2\beta}} +1)\right)^{-\frac 1p} \, \|f\|_{L^q}.\]
\eqref{lpexp3} {Using \eqref{exp-l-p-est} and \eqref{pl1}, we obtain }
\[\|e^{-tH^{\beta}}f\|_{\exp L^p} \leq \frac{1}{(\log 2)^{\frac 1p}}\left( \|e^{-tH^{\beta}}f\|_{L^q}+\|e^{-tH^{\beta}}f\|_{L^\infty}\right).\]
Owing to Theorem \ref{dc}, we obtain
\[\|e^{-tH^{\beta}}f\|_{\exp L^p} \leq \frac{{C}}{(\log 2)^{\frac 1p}}\left( \|f\|_{L^q}+t^{-\frac{d}{2\beta r}}\|f\|_{L^r}\right).\]
\end{proof}

{To establish local well-posedness, we also need a Lebesgue estimate for normalized Hermite functions, along with a smoothing property, as outlined below.
\begin{Lemma}
    \label{Hermite}
   Let \(\Phi_{\alpha}\), where \(\alpha \in \mathbb{N}^d\), denote the normalized Hermite functions. Then, the following holds:
   \begin{equation}
\label{Herm-est}
\|\Phi_{\alpha}\|_{L^q} \leq C\,(1+|\alpha|)^{\frac d4},\quad \forall\; 1\leq q\leq \infty,\;\alpha\in\N^d.
\end{equation}
\end{Lemma}}
\begin{proof}
{We emphasize that \eqref{Herm-est} was originally established in \cite{thangavelu1993lectures} for the one-dimensional case ($d=1$). The $d$-dimensional estimate \eqref{Herm-est} follows directly by observing that the functions $\Phi_{\alpha}$, where $\alpha \in \mathbb{N}^d$, are tensor products of one-dimensional Hermite functions. Specifically, for \(x = (x_{1}, x_{2}, \dots, x_{d})\) and \(\alpha \in \mathbb{N}^{d}\) with \(|\alpha| = \alpha_{1} + \dots + \alpha_{d}\), we have
\[
\Phi_{\alpha}(x) = \prod^{d}_{j=1} h_{\alpha_{j}}(x_{j}),
\]
where
\[
h_{k}(z) = \left(\sqrt{\pi} 2^{k} k!\right)^{-1/2} (-1)^{k} e^{\frac{z^2}{2}} \frac{d^{k}}{dz^{k}} e^{-z^2}.
\]
From \cite[Lemma 1.5.2]{thangavelu1993lectures}, we obtain the estimate
\[
\|h_{k}\|_{L^{q}} \leq C \left(1 + k\right)^{\frac{1}{4}}, \quad \forall\; 1 \leq q \leq \infty.
\]
Consequently, we derive the following estimate for $\Phi_{\alpha}$:
\[
\|\Phi_{\alpha}\|_{L^{q}} \leq C \prod_{i=1}^{d} (1 + \alpha_{i})^{\frac{1}{4}} \leq C\left(1 + |\alpha|\right)^{\frac{d}{4}}, \quad \forall\; 1 \leq q \leq \infty, \; \alpha \in \mathbb{N}^{d}.
\]
The last inequality follows from the fact that \(\alpha_{i} \leq |\alpha|\) for each \(i\). 
This completes the proof of Lemma \ref{Hermite}.}
\end{proof}
\begin{proposition} \label{conexp0} {Let $1\leq p < \infty$ and $\beta >0.$}
If $g \in \exp L^p_0(\R^d)$, then $$e^{-tH^{\beta}} g \in C([0,\infty), \exp L^p_0(\R^d)).$$
\end{proposition}
\begin{proof}
The proof of this result uses similar idea as in \cite{majdoub2018ade, fino2020cpaa}. {By density of $C_0^\infty(\R^d)$ in $\exp L^p_0(\R^d)$, it suffices to show that
\begin{equation}
\label{Density}
\lim_{t\to 0}\left\| e^{-tH^{\beta}} g-g\right\|_{L^q}=0
\end{equation}
for all $g\in C_0^\infty(\R^d)$ and $1\leq q \leq \infty$.} We note that 
\[e^{-t H^{\beta}}g(x)=   \sum_{k=0}^\infty e^{-t(2k+d)^{\beta}} P_kg(x),~~ P_kg=\sum_{|\alpha|=k}\langle g,\Phi_{\alpha} \rangle \, \Phi_{\alpha},\]
where $\Phi_{\alpha},~\alpha \in \mathbb{N}^d$, are the normalized Hermite functions.  Now, since $H \Phi_{\alpha}=(2|\alpha|+d) \, \Phi_{\alpha}$, an integration by parts yields 
\[\langle g, \Phi_{\alpha}\rangle=(d+2|\alpha|)^{-N} \, \langle H^Ng, \Phi_{\alpha}\rangle, \; N\in \N.\]
{Since \(H^N g \in C_0^\infty(\mathbb{R}^d)\), we can apply Hölder's inequality along with \eqref{Herm-est} to obtain} \[|\langle g, \Phi_{\alpha}\rangle|\leq C (d+2|\alpha|)^{-N+\frac d4}\,\|H^N\,g\|_{L^{q'}}.\]

Hence 
\begin{eqnarray*}
\|P_k g\|_{L^q} &\leq& C\sum_{|\alpha|=k} \,(d+2|\alpha|)^{-N+\frac{d}{2}}\|H^N\,g\|_{L^{q'}}\\
&\leq& C (d+k)^{-N+\frac{d}{2}} \sum_{|\alpha|=k} 1\\
&\leq& C(d+k)^{-N+\frac{3d}{2}-1},
\end{eqnarray*}
where we have used the fact that
$$
\sum_{|\alpha|=k} 1=\begin{pmatrix}
           k+d-1 \\
           k 
         \end{pmatrix}
         \lesssim (d+k)^{d-1}.
$$

Since $g\in C_0^\infty(\R^d)$, we get
\begin{eqnarray} \label{E2.6}
    \left\| e^{-tH^{\beta}} g-g\right\|_{L^q}&=&\left\| \sum_{k=0}^\infty [e^{-t(2k+d)^{\beta}} P_kg-P_kg]\right\|_{L^q} \notag\\
    &\leq& {\sum_{k=0}^{\infty} \left(1-e^{-t(2k+d)^{\beta}}\right)\left\| P_kg\right\|_{L^q}}\notag\\
    &\leq& C \sum_{k=0}^{\infty} \left(1-e^{-t(2k+d)^{\beta}}\right) \, (d+k)^{{-N+\frac{3d}{2}-1}}.
\end{eqnarray}
Therefore, by taking $N$ large enough and the limit as $t\to 0$ in \eqref{E2.6}, we infer
\[\lim_{t\to 0}\left\| e^{-tH^{\beta}} g-g\right\|_{L^q}=0.\] This completes the proof.
\end{proof}

As a consequence, we have the following:
\begin{Corollary} \label{C2.3}
    Let ${0<\beta\leq 1}, 
  ~p>1,~d>\frac{2\beta p}{p-1},~r>\frac d{2\beta}.$  Then, for every $g\in L^1 (\mathbb R^d) \cap L^r (\mathbb R^d)$, we have
    \[\|e^{-tH^{\beta}} g\|_{\exp L^p} \leq \kappa(t)   \, [\|g\|_{L^1 }+\|g\|_{L^r}],~\forall\, t>0,\]
    where $\kappa \in L^1(0,\infty)$ is given by
    \[\kappa(t)=\frac{{C}} {(\log 2)^{\frac 1p}} \, \min \left\{t^{-\frac{d}{2\beta r}}+1,t^{-\frac{d}{2\beta}}(\log(t^{-\frac{d}{2\beta}}+1))^{-\frac 1p} \right\}.\]
\end{Corollary}
\begin{proof}
 By Proposition \ref{lpexp} $(2)$ with $q=1$,  we have
 \begin{align} \label{E1}
 \|e^{-tH^{\beta}} g\|_{\exp L^p} \leq C  t^{-\frac{d}{2\beta}} \, \left( \log (t^{-\frac d{2\beta}} +1)\right)^{-\frac 1p} \, \|g\|_{L^1}.
 \end{align}
 On the other hand, by Proposition \ref{lpexp} $(3)$ with $q=1$, we obtain
 \begin{align} \label{E2}
\|e^{-tH^{\beta}} g\|_{\exp L^p}& \leq \frac{{C}}{(\log 2)^{\frac 1p}}\left[ t^{-\frac{d}{2 \beta r}} \|g\|_{L^r} +\|g\|_{L^1}\right],\notag\\
&\leq\frac{{C}}{(\log 2)^{\frac 1p}} (t^{-\frac{d}{2 \beta r}}+1) \left[\|g\|_{L^r} +\|g\|_{L^1}\right].
 \end{align}
Combining \eqref{E1} and \eqref{E2}, we obtain
\[\|e^{-tH^{\beta}} g\|_{\exp L^p} \leq \kappa(t)   \, \left[\|g\|_{L^1 }+\|g\|_{L^r}\right],~\forall\, t>0.\]
Thanks to the assumptions $d>\frac{2\beta p}{p-1}$ and $r>\frac d{2\beta}$, we see that $\kappa \in L^1(0,\infty)$.
\end{proof}
For $d=\frac{2\beta p}{p-1}$, we also have similar result in some suitable Orlicz space. Let $\phi(s):=e^{s^p}-1-s^p,~s\geq 0$ and $L^{\phi}$ be the associated Orlicz space endowed with the Luxemburg norm \eqref{Luxnorm}. From the definition, we have 
\begin{align} \label{E3.2}
C_1 \|g\|_{\exp L^p} \leq \|g\|_{L^p}+\|g\|_{L^{\phi}} \leq C_2 \|g\|_{\exp L^p},
\end{align}
for some $C_1,~C_2>0.$
\begin{Corollary} \label{C3.2}
    Let ${0<\beta\leq 1}, 
  ~p>1,~r>\frac d{2\beta}=\frac{p}{p-1}.$ For every $g \in L^1 (\mathbb R^d) \cap L^{2p} (\mathbb R^d)\cap L^r (\mathbb R^d),$ we have 
    \[\|e^{-tH^{\beta}} g\|_{ L^{\phi}} \leq \zeta(t)   \, \left[\|g\|_{L^1 }+ \|g\|_{L^{2p}}+\|g\|_{L^r}\right],~\forall ~t>0,\]
    where $\zeta \in L^1(0,\infty)$ is given by
    \[\zeta(t)=\frac{{C}}{(\log 2)^{\frac 1p}} \, \min \left\{t^{-\frac{d}{2\beta r}}+1,t^{-\frac{p}{p-1}}(\log(t^{-\frac{p}{p-1}}+1))^{-\frac 1{2p}} \right\}.\]
\end{Corollary}
\begin{proof}
 In view of Proposition \ref{lpexp}, we obtain
 \begin{align*}
     \int_{\R^d} \phi\left(\frac{|e^{-tH^{\beta}}g|}{\lambda}\right) \, dx &=\sum_{k\geq 2} \frac{\|e^{-tH^{\beta}}g\|_{L^{pk}}^{pk}}{\lambda^{pk} k!}\\
     &\leq \sum_{k\geq 2} \frac{{C^{pk}}t^{-\frac d{2\beta}(1-\frac 1{pk})pk}\|g\|_{L^1}^{pk}}{\lambda^{pk} k!}\\
     &=\sum_{k\geq 2} \frac{{C^{pk}}t^{-\frac {p}{p-1}(1-\frac 1{pk})pk}\|g\|_{L^1}^{pk}}{\lambda^{pk} k!}=t^{\frac p{p-1}} \, \phi\left({C}t^{-\frac p{p-1}}\frac{\|g\|_{L^1}}{\lambda}\right)\\
     &\leq t^{\frac p{p-1}} \, \left(\exp \left\{\left({C}t^{-\frac p{p-1}}\frac{\|g\|_{L^1}}{\lambda}\right)^{2p}\right\}-1\right).
 \end{align*}
In the last step we have used the fact that $e^s-1-s\leq e^{s^2}-1$ for every $s\geq 0.$ Thus we obtain that
 \begin{align*}
     \|e^{-tH^{\beta}} g\|_{ L^{\phi}}&\leq \inf \left\{\lambda>0: t^{\frac p{p-1}} \, \left(\exp \left\{\left({C}t^{-\frac p{p-1}}\frac{\|g\|_{L^1}}{\lambda}\right)^{2p}\right\}-1\right) \leq 1\right\}\\
     &={C}t^{-\frac p{p-1}} \left(\log (t^{-\frac p{p-1}}+1)\right)^{-\frac 1{2p}} \|g\|_{L^1}.
 \end{align*}
 In view of the embedding $L^{2p}\cap L^{\infty} \rightarrow L^{\phi},$ we also have 
 \[\|e^{-tH^{\beta}} g\|_{ L^{\phi}} \leq (\log 2)^{-\frac 1p}\left[\|e^{-tH^{\beta}} g\|_{L^{\infty}}+\|e^{-tH^{\beta}} g\|_{L^{2p}}\right].\]
 By utilizing Proposition \ref{lpexp} and selecting $r>\frac d{2\beta}=\frac p{p-1}$, we deduce that
 \[\|e^{-tH^{\beta}} g\|_{ L^{\phi}} \leq (\log 2)^{-\frac 1p}\left[t^{-\frac d{2\beta r}}\| g\|_{L^{r}}+\| g\|_{L^{2p}}\right].\]
 Combining above inequalities, we get
\[\|e^{-tH^{\beta}} g\|_{ L^{\phi}} \leq \zeta(t)   \, [\|g\|_{L^1 }+ \|g\|_{L^{2p}}+\|g\|_{L^r}],~\forall ~ t>0.\]
Since $\frac{d}{2\beta r}<1$ and $\frac p{p-1}-\frac p{p-1} \frac 1{2p}=\frac{2p}{2(p-1)}>1$, we see that $\zeta \in L^1(0,\infty)$.
\end{proof}
\begin{Lemma}(\cite[Lemma  4.1.5]{cazenave1990introduction}). \label{L3.3}
Let $X$ be a Banach space and $g \in L^1(0,T;X).$ Then, for any $\beta>0$, we have $$t\longmapsto \int_0^t\,e^{-(t-\tau)H^{\beta}} \, g(\tau) \, d\tau \in C([0,T];X).$$
\end{Lemma}
\begin{proposition}(\cite[Proposition 2.9]{majdoub2018picm}).
 \label{pnlex}
Let $1\leq p <\infty$ and $u \in C([0,T]; \exp L^p_0(\R^d))$ for some $T>0.$ Then, for every $\lambda>0,$ we have
\[(e^{\lambda |u|^p}-1) \in C([0,T], L^r(\R^d)), ~1\leq r<\infty.\]
\end{proposition}

\begin{Corollary}[\cite{majdoub2018picm}]
Let $1\leq p <\infty$ and $u \in C([0,T]; \exp L^p_0(\R^d))$ for some $T>0.$ Assume that $f$ satisfies \eqref{nflw}. Then, for every $p\leq r<\infty,$ we have
\[f(u) \in C([0,T];L^r(\R^d)).\]
\end{Corollary}

To prove the global existence results, the following estimate of the nonlinear term will be handy later.
\begin{Lemma}(\cite[Lemma 2.6, p. 2387]{majdoub2018picm}). \label{L2.31}
 Let $\lambda>0,~1\leq p,\, q<\infty$ and $K>0$ such that $\lambda q K^p\leq 1$. Assume that 
 \[\|u\|_{\exp L^p} \leq K.\]
 Then \[\|e^{\lambda |u|^p}-1\|_{L^q}\leq (\lambda q K^p)^{\frac 1q}.\]
\end{Lemma}

\begin{Lemma}(\cite[Lemma 2.6]{majdoub2018picm} and \cite{fino2020cpaa}).
\label{L2.3} 
    Let $m\geq p>1,~a>\frac{p(m-1)}{p-1}.$ Define $\sigma=\frac 1{m-1}-\frac d{2\beta a}.$ Assume that $d>\frac{2\beta p}{p-1},~a<\frac{d(m-1)}{2\beta} \frac{1}{(2-m)_+}$. Then there exist $r,q,\{\theta_k\}_{k=0}^{\infty}, \{\rho_k\}_{k=0}^{\infty}$ such that $1 < r\leq a,~q\geq 1$ and $\frac 1r=\frac 1a+\frac 1q,~$ $0<\theta_k<1$ and $\frac{1}{q(pk+m-1)}=\frac{\theta_k}{a}+\frac{1-\theta_k}{\rho_k},~ p\leq \rho_k<\infty,~\frac{d}{2\beta}(\frac 1r-\frac 1a)<1$, 
    
    $$\sigma[1
    +\theta_k(pk+m-1)]<1.$$
    
    \[1-\frac{d}{2\beta}\left(\frac 1r-\frac 1a\right)-\sigma \theta_k(pk+m-1)=0.\]
\end{Lemma}
{To conclude this section, we state the Banach fixed point theorem, a key tool that will play a central role in our analysis.
\begin{theorem}[{\sf Banach Fixed Point Theorem}]
\label{BFP}
Let \((\mathbf{X}, d)\) be a non-empty complete metric space, and let \(\Phi: \mathbf{X} \to \mathbf{X}\) be a contraction mapping, i.e., there exists a constant \(0 \leq \kappa < 1\) such that  
\[
d(\Phi(x), \Phi(y)) \leq \kappa \, d(x, y) \quad \text{for all } x, y \in \mathbf{X}.
\]  
Then, the following hold:
\begin{enumerate}
    \item[1)] \(\Phi\) has a unique fixed point \(x^* \in \mathbf{X}\) satisfying \(\Phi(x^*) = x^*\).
    \item[2)] For any initial point \(x_0 \in \mathbf{X}\), the sequence defined by \(x_{n+1} = \Phi(x_n)\) converges to \(x^*\) as \(n \to \infty\).
\end{enumerate}
\end{theorem}
}

\section{Proof of Theorem \ref{T1.1}}
\label{S3}
Our approach closely follows the proof method outlined in \cite[Theorem 1.3]{fino2020cpaa}, building on earlier works such as \cite{majdoub2018picm, majdoub2018ade, ioku2015MPAG}. For brevity, we provide only a brief sketch of the proof.

The main idea is to decompose the initial data  $u_0 \in \exp L^p_0(\R^d),$ using the density of $C_0^{\infty}(\R^d),$ into a small part in $\exp L^p(\R^d)$ and a smooth one.
 Let $u_0\in \exp L^p_0(\R^d).$ Then by density,  for every $\epsilon>0$ there exists $v_0\in C_0^{\infty}(\R^d)$ such that $u_0=v_0+w_0$ with
\[\|w_0\|_{\exp L^p(\R^d)} \leq \epsilon.\]
In order to study the problem \eqref{hexp}, we consider the following two problems: 
\begin{eqnarray}\label{hexp1}
\begin{cases} \partial_tv + (-\Delta +|x|^2)^{\beta}v=f(v),\\
v(x,0)=v_0\in C_0^{\infty}(\R^d),
\end{cases} 
\end{eqnarray}
and
\begin{eqnarray}\label{hexp2}
\begin{cases} \partial_tw + (-\Delta +|x|^2)^{\beta}w=f(w+v)-f(v),\\
w(x,0)=w_0,~\|w_0\|_{\exp L^p} \leq \epsilon.
\end{cases} 
\end{eqnarray}
We observe that when $v$ and $w$ are mild solutions of \eqref{hexp1} and \eqref{hexp2} respectively, the function $u=v+w$ satisfies \eqref{hexp} as a mild solution. We will now establish the local well-posedness for \eqref{hexp1} and \eqref{hexp2}.

\begin{Lemma} \label{L4.1}
Let $v_0 \in  L^p(\R^d) \cap L^{\infty}(\R^d),  p>1,  \beta>0$. Assume that $f$ satisfies \eqref{nflw}. Then there exists $T= T(v_0)>0$   and a mild solution  $v \in C([0, T]; \exp L^p_0(\R^d)) \cap L^{\infty}(0,T;L^{\infty}(\R^d))$ of \eqref{hexp1}.
\end{Lemma}

\begin{Lemma} \label{L4.2}
Let $w_0 \in  \exp L^p_0(\R^d),  p>1,  \beta>0$. Assume that $f$ satisfies \eqref{nflw}. Let $T>0$ and $v\in L^{\infty}(0,T;L^{\infty}(\R^d))$ be given in Lemma \ref{L4.1}. Then for $\|w_0\|_{\exp L^p} \leq \epsilon$ with $\epsilon \ll 1$ small enough, there exists $\tilde{T}= \tilde{T}(w_0,\epsilon,v)>0$ and a mild solution  $w \in C([0, \tilde{T}]; \exp L^p_0(\R^d))$ of \eqref{hexp2}.
\end{Lemma}
To prove the lemmas mentioned above, we require the following result.
\begin{Lemma}\cite[Lemma 4.4]{majdoub2018picm} \label{L4.3} 
Let $v\in L^{\infty}(0,T;L^{\infty}(\R^d))$ for some $T>0$. Let $1<p\leq q<\infty,$ and $w_1,w_2\in \exp L^p(\R^d)$ with $\|w_1\|_{\exp L^p},~\|w_2\|_{\exp L^p}\leq M$ for sufficiently small $M>0$ (namely $2^p \lambda qM^p\leq 1$, where $\lambda$ is given as in \eqref{nflw}). Then there {exists} a constant $C_q>0$ such that
\[\|f(w_1+v)-f(w_2+v)\|_{L^q}\leq C_q e^{2^{p-1}\lambda \|v\|_{\infty}^p} \, \|w_1-w_2\|_{\exp L^p}.\]
\end{Lemma}

\begin{proof}[Proof of Lemma \ref{L4.1}]
We consider
{\[Y_T:=\Big\{v\in C([0, T], \exp L^p_0(\R^d)) \cap L^{\infty}(0,T;L^{\infty}(\R^d)): \|v\|_{Y_T}\leq 2\|v_0\|_{L^p\cap L^{\infty}}\Big\},\]}
where $\|v\|_{Y_T}:=\|v\|_{L^{\infty}(0,T;L^p)}+\|v\|_{L^{\infty}(0,T;L^{\infty})}$ and $\|v_0\|_{L^p\cap L^{\infty}}:=\|v_0\|_{L^p}+\|v_0\|_{L^{\infty}}.$
Put
\[\Phi(v):=e^{-tH^{\beta}} v_0+\int_0^t e^{-(t-\tau)H^{\beta}} \, f(v(\tau)) \, d\tau.\]
 
{By combining \eqref{exp-l-p-est}, \eqref{pl2}, Proposition \ref{conexp0}, Theorem \ref{dc}, and Lemma \ref{L3.3}, it follows that $\Phi$ maps $Y_T$ into itself. Furthermore, using Theorem \ref{dc} and adopting the reasoning from \cite[Lemma 3.1]{fino2020cpaa}, we can conclude that $\Phi$ is a contraction mapping for sufficiently small $T > 0$. Applying Theorem \ref{BFP} leads to the desired result.}
\end{proof}
\begin{proof}[Proof of Lemma \ref{L4.2}]
For  $\tilde{T}>0,$ we  consider
\[W_{\tilde{T}}=\left\{ w\in C([0,\tilde{T}], \exp L^p_0(\R^d))): \|w\|_{L^{\infty}(0,\tilde{T};\exp L^p_0)} \leq 2\epsilon \right\}.\] Put
\[\tilde{\Phi}(w):=e^{-tH^{\beta}} w_0+\int_0^t e^{-(t-\tau)H^{\beta}} \, [f(w(\tau)+v(\tau))-f(v(\tau))] \, d\tau.\]
We shall prove that $\tilde{\Phi}:W_{\tilde{T}} \to W_{\tilde{T}}$ is a contraction map for sufficiently  small  $\epsilon$ and $\tilde{T}>0.$ To do that let $w_1,w_2\in W_{\tilde{T}}.$  {By \eqref{exp-l-p-est} and \eqref{pl1}, we get}
\[\|\tilde{\Phi}(w_1)-\tilde{\Phi}(w_2)\|_{\exp L^p} \leq \frac{{C}}{(\ln 2)^{\frac 1p}} \Big(\|\tilde{\Phi}(w_1)-\tilde{\Phi}(w_2)\|_{L^p}+\|\tilde{\Phi}(w_1)-\tilde{\Phi}(w_2)\|_{ L^{\infty}}\Big).\]
In view of Theorem \ref{dc} and by invoking Lemma \ref{L4.3}, we obtain
\[\|\tilde{\Phi}(w_1)-\tilde{\Phi}(w_2)\|_{ L^{\infty}} \leq C e^{2^{p-1}\lambda \|v\|_{L^\infty}^p} \tilde{T}^{1-\frac{d}{2\beta r}} \, \|w_1-w_2\|_{L^{\infty}(0,\tilde{T}:\exp L^p)},\]
where $r>0$ is an arbitrary constant such that $r>\max\{p,\frac{d}{2\beta}\}$ and $2^p\lambda r (2\epsilon)^p\leq 1.$
Similarly, we also obtain
\[\|\tilde{\Phi}(w_1)-\tilde{\Phi}(w_2)\|_{L^p}\leq C e^{2^{p-1}\lambda \|v\|_{\infty}^p} \tilde{T} \, \|w_1-w_2\|_{L^{\infty}(0,\tilde{T}, \exp L^p)}.\]
Thus by choosing $\epsilon \ll 1$ small, we infer that
\begin{align*}
\|\tilde{\Phi}(w_1)-\tilde{\Phi}(w_2)\|_{\exp L^p} &\leq C e^{2^{p-1}\lambda \|v\|_{\infty}^p} (\tilde{T}+\tilde{T}^{1-\frac{d}{2\beta r}}) \, \|w_1-w_2\|_{L^{\infty}(0,\tilde{T}:\exp L^p)}\\
&\leq \frac 12 \|w_1-w_2\|_{L^{\infty}(0,\tilde{T}:\exp L^p)},
\end{align*}
where $\tilde{T}\ll 1$ is chosen small enough such that $C e^{2^{p-1}\lambda \|v\|_{\infty}^p} (\tilde{T}+\tilde{T}^{1-\frac{d}{2\beta r}}) \leq \frac 12.$ By considering Propositions \ref{lpexp}-\ref{conexp0}, and following the arguments presented in the proof of \cite[Lemma 3.2]{fino2020cpaa}, we obtain the desired result.
\end{proof}
\begin{proof}[Proof of Theorem \ref{T1.1}] To establish Theorem \ref{T1.1}, we will utilize the results presented in Lemma \ref{L4.1} and Lemma \ref{L4.2}.
We choose $T, \epsilon,$ and $\tilde{T}$ in the following way. Let $r>\max\{p, \frac{d}{2\beta}\}$ and fix $\epsilon>0$ such that $2^p \lambda r (2\epsilon)^p\leq 1.$ In order to use Lemma \ref{L4.1}, we first decompose $u_0=v_0+w_0$ with $v_0\in C_0^{\infty}(\R^d)$ and $\|w_0\|_{\exp L^p}\leq \epsilon$ as before. Then by Lemma \ref{L4.1}, there exist a time $T>0$ and a mild solution $v\in C([0, T], \exp L^p_0(\R^d)) \cap L^{\infty}(0,T;L^{\infty}(\R^d))$ of \eqref{hexp1} such that 
$$ \|v\|_{L^{\infty}(0,T;L^p\cap L^{\infty})} \leq 2\|v_0\|_{L^p\cap L^{\infty}}.$$ 
Next we choose $\tilde{T}>0$ such that $\tilde{T}<T$ and 
\[C e^{2^{p-1}\lambda \|v\|_{L^p \cap L^\infty}^p} (\tilde{T}+\tilde{T}^{1-\frac{d}{2\beta r}}) \leq \frac 12.\]
Then by Lemma \ref{L4.2}, there exists a mild solution  $w \in C([0, \tilde{T}], \exp L^p_0(\R^d))$ of \eqref{hexp2}. Hence $u:=v+w$ is  a mild solution of \eqref{hexp} in $C([0, \tilde{T}], \exp L^p_0(\R^d))$.  This proves the existence part.  By incorporating Theorem \ref{dc} and Proposition \ref{pnlex}, and closely following the proof methodology utilized in \cite[Theorem 1.3]{fino2020cpaa}, we can establish the uniqueness. Therefore, for brevity, we will omit the detailed explanation.
\end{proof}

\section{Proof of Theorem \ref{Th1}}
\label{S4}
\begin{proof}[Proof of  Theorem \ref{Th1}] $(1):$ 
We closely follow the approach introduced in \cite[Theorem 1.3]{fino2020cpaa} and draw inspiration from  \cite{majdoub2018picm, majdoub2018ade, ioku2015MPAG}. Specifically, we consider the associated integral equation
\begin{align} \label{E4.1}
u(t)=e^{-tH^{\beta}} u_0+\int_0^t e^{-(t-s)H^{\beta}} \, f(u(s)) \, ds
\end{align}
where $\|u_0\|_{\exp L^p}\leq \epsilon,$ with small $\epsilon>0$ to be fixed later. The nonlinearity $f$ satisfies $f(0)=0$
and \begin{align} \label{E4.2} 
|f(u)-f(v)|\leq C |u-v|\left(|u|^{m-1}e^{\lambda |u|^p}+|v|^{m-1} e^{\lambda |v|^p}\right),
\end{align}
for some constants $C>0$ and $\lambda>0$. Here $p>1$ and $m$ is larger than $1+\frac{2 p \beta}{d}$.  
From \eqref{E4.2}, we see that
\begin{align}\label{E4.3}
  |f(u)-f(v)|\leq C |u-v| \sum_{k=0}^{\infty} \frac{\lambda^k}{k!}\left(|u|^{pk+m-1}+|v|^{pk+m-1}\right).  
\end{align}
First, we focus on the proof of Theorem \ref{Th1} (1). 

Let $M>0$ and  $$Y_M=\left\{u \in L^{\infty}\left( (0,\infty), \exp L^p(\R^d)\right) : \sup_{t>0} t^{\sigma} \|u(t)\|_{L^a}+\|u\|_{L^{\infty}\left((0,\infty), \exp L^{p}(\R^d)\right)}\leq M\right\},$$
where $a>\frac{d(m-1)}{2 \beta}\geq p$ and $\sigma=\frac{1}{(m-1)}-\frac d{2\beta a}=\frac d{2\beta}(\frac{2\beta}{d(m-1)}-\frac 1a)>0.$

Let $\rho(u,v)=\displaystyle\sup_{t>0} \left(t^{\sigma} \|u(t)-v(t)\|_{L^a}\right).$ It is easy to see that $(Y_M,\rho)$ is a complete metric space.

Now, we  define the function $\Phi$ on $Y_M$ as follows
\begin{align} \label{Emap}
\Phi[u](t)=e^{-tH^\beta}u_0+\int_0^t e^{-(t-\tau)H^\beta}\left(f(u(\tau)) \right) \, d\tau.  
\end{align}

{By Proposition \ref{lpexp} \eqref{lpexp1}, Theorem \ref{dc} and \eqref{L-q-exp-L-p}, we have}
\begin{align} \label{E4.5}
\|e^{-tH^{\beta}}u_0\|_{\exp L^p} \leq {C}\|u_0\|_{\exp L^p},
\end{align}
and 
\begin{align} \label{E4.6}
t^{\sigma} \|e^{-tH^{\beta}}u_0\|_{L^a} \leq {C}t^{\sigma} t^{-\frac d{2\beta} (\frac{2\beta}{d(m-1)}-\frac 1a)} \|u_0\|_{L^{\frac{d(m-1)}{2\beta}}}={C}\|u_0\|_{L^{\frac{d(m-1)}{2\beta}}}\leq {C}\|u_0\|_{\exp L^{p}},
\end{align}
where we have used $1<p\leq \frac{d(m-1)}{2\beta}<a.$ 

In the subsequent analysis, we will consider and address the cases where $d>\frac{2\beta p}{p-1}$, $d=\frac{2\beta p}{p-1}$, and $d<\frac{2\beta p}{p-1}$ separately.

\subsection{The case $d>\frac{2\beta p}{p-1}$}
Let $u \in Y_{M }$. Then by Proposition \ref{lpexp} and Corollary \ref{C2.3}, we obtain for $q>\frac d{2\beta},$
\begin{eqnarray*}
\|\Phi(u)(t)\|_{\exp L^p} 
&\leq& \|e^{-tH^{\beta}}u_0\|_{\exp L^p}+\left\|\int_0^t e^{-(t-\tau)H^\beta}\left(f(u(\tau))\right) \, d\tau\right\|_{\exp L^p}\\
&\leq& {C}\|u_0\|_{\exp L^p}+\int_0^t \left\|e^{-(t-\tau)H^\beta}\left(f(u(\tau))\right) \, \right\|_{\exp L^p} d\tau\\
& \leq  &{C}\|u_0\|_{\exp L^p}+   \int_0^t \kappa(t-\tau) \,  \|f(u(\tau)\|_{L^1\cap L^q}  d\tau\\
& \leq  &{C}\|u_0\|_{\exp L^p}+  \,  \|f(u(\tau)\|_{L^{\infty}(0,\infty;(L^1\cap L^q)} \,  \int_0^{\infty} \kappa(\tau) d\tau\\
& \leq  &  {C}\|u_0\|_{\exp L^p}+ C   \,  \|f(u(\tau)\|_{L^{\infty}(0,\infty;(L^1\cap L^q)},
\end{eqnarray*}
where $\kappa(\tau)$ is as in Corollary \ref{C2.3}.

By \eqref{nfgw}, we have 
\begin{align*}
    |f(u)| \leq C|u|^m(e^{\lambda |u|^p}-1) +C|u|^m,~m\geq p.
\end{align*}
{By applying H\"older's inequality and \eqref{L-q-exp-L-p}, for \(1 \leq r \leq q\) and \(m \geq p\), we obtain}
\begin{align}
 \|f(u)\|_{L^r} \leq C \|u\|^m_{\exp L^p} (\|e^{\lambda |u|^p}-1\|_{L^{2r}}+1). 
\end{align}
By applying Lemma \ref{L2.31} and considering the fact that $u\in Y_M$, we can deduce the following inequality when $2q\lambda M^p\leq 1$
\begin{align}
    \|f(u)\|_{L^{\infty}(0,\infty; L^r)} \leq C M^m.
\end{align}
Finally, we obtain that
\begin{align*}
\|\Phi(u)\|_{L^{\infty}(0,\infty,\exp L^p)} \leq {C}\|u_0\|_{\exp L^p} +C M^m\leq {C}\epsilon +C M^m.
\end{align*}
Let $u,v$ be two elements of $Y_M.$ By using \eqref{E4.3} and Proposition \ref{lpexp}, one gets
\begin{align} \label{E4.100}
t^{\sigma}\|\Phi(u)(t)-\Phi(v)(t)\|_{L^a}\leq C\rho(u,v) \sum_{k=0}^{\infty}(C\lambda)^k M^{pk+m-1}.
\end{align}
Indeed,
\begin{align*}
& t^{\sigma}\|\Phi(u)(t)-\Phi(v)(t)\|_{L^a}\leq t^{\sigma} \int_0^t \left\|e^{-(t-\tau)H^{\beta}}(f(u(\tau))-f(v(\tau))) \, \right\|_{L^{a}} \, d\tau \\
   &\leq C \sum_{k=0}^{\infty} \frac{\lambda^k}{k!} t^{\sigma} \int_0^t (t-\tau)^{-\frac d{2\beta}(\frac 1r-\frac 1a)} \|(u-v)(|u|^{pk+m-1}+|v|^{pk+m-1})\|_{L^r} \, d\tau,
\end{align*}
$1\leq r\leq a.$
Applying  H\"{o}lder's inequality, we obtain
\begin{align*}
 & t^{\sigma}\|\Phi(u)(t)-\Phi(v)(t)\|_{L^a} \\
   &\leq C \sum_{k=0}^{\infty} \frac{\lambda^k}{k!} t^{\sigma} \int_0^t (t-\tau)^{-\frac d{2\beta}(\frac 1r-\frac 1a)} \|(u-v)\|_{L^a}\|(|u|^{pk+m-1}+|v|^{pk+m-1})\|_{L^q} \, d\tau\\
   &\leq C \sum_{k=0}^{\infty} \frac{\lambda^k}{k!} t^{\sigma} \int_0^t (t-\tau)^{-\frac d{2\beta}(\frac 1r-\frac 1a)} \|(u-v)\|_{L^a}[\|u\|_{L^{q(pk+m-1)}}^{pk+m-1}+\|v\|_{L^{q(pk+m-1)}}^{pk+m-1}] \, d\tau.
\end{align*}
Using interpolation inequality with $\frac{1}{q(pk+m-1)}=\frac{\theta}{a}+\frac{1-\theta}{\rho},~0\leq \theta \leq 1$ and $p\leq \rho<\infty,$ we find that
\begin{align*}
 & t^{\sigma}\|\Phi(u)(t)-\Phi(v)(t)\|_{L^a} \\
   &\leq C \sum_{k=0}^{\infty} \frac{\lambda^k}{k!} t^{\sigma} \int_0^t (t-\tau)^{-\frac d{2\beta}(\frac 1r-\frac 1a)} \|(u-v)\|_{L^a}\\
   &\left[\|u\|_{L^{a}}^{(pk+m-1)\theta} \, \|u\|_{L^{\rho}}^{(pk+m-1)(1-\theta)}+\|v\|_{L^{a}}^{(pk+m-1)\theta} \, \|v\|_{L^{\rho}}^{(pk+m-1)(1-\theta)}\right] \, d\tau.
\end{align*}
Owing to Lemma \ref{relaorL}, we infer
\begin{align*}
 & t^{\sigma}\|\Phi(u)(t)-\Phi(v)(t)\|_{L^a} \\
   &\leq C \sum_{k=0}^{\infty} \frac{\lambda^k}{k!} t^{\sigma} \int_0^t (t-\tau)^{-\frac d{2\beta}(\frac 1r-\frac 1a)} \|(u-v)\|_{L^a} \Gamma\left(\frac{\rho}{p}+1\right)^{\frac{(pk+m-1)(1-\theta)}{\rho}}\\
   &\left[\|u\|_{L^{a}}^{(pk+m-1)\theta} \, \|u\|_{\exp L^p}^{(pk+m-1)(1-\theta)}+\|v\|_{L^{a}}^{(pk+m-1)\theta} \, \|v\|_{\exp L^p}^{(pk+m-1)(1-\theta)}\right] \, d\tau.
\end{align*}
Since $u,v \in Y_M$, we obtain
\begin{align*}
 & t^{\sigma}\|\Phi(u)(t)-\Phi(v)(t)\|_{L^a} \\
   &\leq C \rho(u,v) \, \sum_{k=0}^{\infty} \frac{\lambda^k}{k!} \Gamma\left(\frac{\rho}{p}+1\right)^{\frac{(pk+m-1)(1-\theta)}{\rho}} \, M^{pk+m-1}\\
   &\times t^{\sigma} \left(\int_0^t (t-\tau)^{-\frac d{2\beta}(\frac 1r-\frac 1a)}\, \tau^{-\sigma(1+(pk+m-1)\theta)} \, d\tau \right)\\
   &\leq C \rho(u,v) \, \sum_{k=0}^{\infty} \frac{\lambda^k}{k!} \Gamma\left(\frac{\rho}{p}+1\right)^{\frac{(pk+m-1)(1-\theta)}{\rho}} \, M^{pk+m-1}\\
   &\times \mathcal{B}\left(1-\frac d{2\beta}\left(\frac 1r-\frac 1a\right),1-\sigma(1+(pk+m-1)\theta)\right),
\end{align*}
where the parameters $a,q,r,\theta=\theta_k,\rho=\rho_k$ are given in Lemma \ref{L2.3}. For these parameters one see that
\[\mathcal{B}\left(1-\frac d{2\beta}\left(\frac 1r-\frac 1a\right),1-\sigma(1+(pk+m-1)\theta)\right) \leq C\]
and 
\[\Gamma\left(\frac{\rho_k}{p}+1\right)^{\frac{(pk+m-1)(1-\theta_k)}{\rho_k}} \leq C^k \, k!.\]
Combining the above estimates we obtain \eqref{E4.100}.
Hence, we get for $M$ small,
\[\rho(\Phi(u),\Phi(v)) \leq C M^{m-1} \, \rho(u,v)\leq \frac 12 \, \rho(u,v).\]
The above estimates show that $\Phi: Y_M \to Y_M$ is a contraction mapping for $\epsilon>0$ and $M$ sufficiently small. 
{By Theorem \ref{BFP}, we conclude that there exists a unique \(u \in Y_M\) such that \(\Phi(u) = u\).} 
By \eqref{Emap}, $u$ solves the integral equation \eqref{E4.1} with $f$ satisfying \eqref{E4.2}. The estimate \eqref{E1.4} follows from $u \in Y_M$. This completes the proof of the existence of a global solution to \eqref{E4.1} for $d > 2\beta p/(p - 1).$
\subsection{The case $d<\frac{2\beta p}{p-1}$}
We will first establish the following two inequalities:
\begin{align}
   \left\|\int_0^t e^{-(t-\tau)H^{\beta}}f(u(\tau)) \, d\tau\right\|_{L^{\infty}(0,\infty: \exp L^p)} \leq C_1(M)
\end{align}
 and 
 \begin{align}\label{Est4.7}
   \sup_{t>0} t^{\sigma} \left\|\int_0^t e^{-(t-\tau)H^{\beta}}(f(u)-f(v)) \, d\tau\right\|_{L^{a}} \leq C_2(M) \, \sup_{\tau>0}(\tau^{\sigma}\|u(\tau)-v(\tau)\|_{L^a}),
\end{align}
where $u,~v \in Y_M$ and $C_1$ and $C_2$ are small when $M$ is small. To prove these estimates, we first note that
\begin{align}
    \left(\log ((t-\tau)^{-\frac d{2\beta}}+1)\right)^{-\frac 1p} \leq 2^{\frac 1p} (t-\tau)^{\frac d{2\beta p}} \text { for } 0\leq \tau<t-\eta^{-\frac {2\beta}d},
\end{align}
where $\eta=\inf\left\{z\geq 1:z>2 \log(1+z)\right\}.$
Thus, by Proposition \ref{lpexp} $(3)$, we obtain for $r>\frac d{2\beta}$ and $0<t\leq \eta^{-\frac {2\beta}d}$,
{\begin{eqnarray*}
    \left\|\int_0^t e^{-(t-\tau)H^{\beta}}f(u(\tau)) \, d\tau\right\|_{\exp L^p} &\leq& C  \int_0^t \, ((t-\tau)^{-\frac d{2 \beta r}}+1) 
    \|f(u(\tau))\|_{L^r\cap L^1} \, d\tau\\ &\leq& C \, \sup_{t>0}\|f(u(\tau))\|_{L^r\cap L^1}.
\end{eqnarray*}}
For $t>\eta^{-\frac {2\beta}d}$ and $1\leq q \leq p,$ we write
\begin{align*}
    &\left\|\int_0^t e^{-(t-\tau)H^{\beta}}f(u(\tau)) \, d\tau\right\|_{\exp L^p} \\
    &\leq   \int_0^{t-\eta^{-\frac {2\beta}d}} \|e^{-(t-\tau)H^{\beta}}f(u(\tau)) \|_{\exp L^p} \, d\tau + \int_{t-\eta^{-\frac {2\beta}d}}^t \|e^{-(t-\tau)H^{\beta}}f(u(\tau)) \|_{\exp L^p} \, d\tau\\
    &\leq \int_0^{t-\eta^{-\frac {2\beta}d}} (t-\tau)^{-\frac d{2\beta q}} (\log((t-s)^{-\frac d{2\beta}}+1))^{-\frac 1p} \|f(u(\tau)) \|_{L^q} \, d\tau \\
    &+ \int_{t-\eta^{-\frac {2\beta}d}}^t ((t-\tau)^{-\frac d{2\beta r}}+1)  \|f(u(\tau)) \|_{L^r \cap L^1} \, d\tau\\
    &\leq C \int_0^{t-\eta^{-\frac {2\beta}d}} (t-\tau)^{-\frac d{2\beta q}+\frac d{2\beta p}}  \|f(u(\tau)) \|_{L^q} \, d\tau +C \sup_{t>0}\|f(u(\tau))\|_{L^r\cap L^1}:=\textbf{I}+\textbf{J}.
\end{align*}
Similar to the analysis of \cite{Mohamed2021, fino2020cpaa}, we obtain for small $M$ and $u \in Y_M$, $\textbf{I}\leq C M^m$ and $\textbf{J}\leq C M^m$.

Finally, we get
\begin{align} \label{E4.23}
\left\|\int_0^t e^{-(t-\tau)H^{\beta}}f(u(\tau)) \, d\tau\right\|_{\exp L^p} \leq C M^m.
\end{align}

To estimate \eqref{Est4.7}, we again use Proposition \ref{lpexp}. This leads to
\begin{align*}
   &t^{\sigma} \left\|\int_0^t e^{-(t-\tau)H^{\beta}}(f(u)-f(v)) \, d\tau\right\|_{L^{a}} \\
   &\leq C \sum_{k=0}^{\infty} \frac{\lambda^k}{k!} t^{\sigma} \int_0^t (t-\tau)^{-\frac d{2\beta}(\frac 1r-\frac 1a)} \|(u-v)(|u|^{pk+m-1}+|v|^{pk+m-1})\|_{L^r} \, d\tau.
\end{align*}
Applying H\"{o}lder's inequality, we infer
\begin{align*}
   &t^{\sigma} \left\|\int_0^t e^{-(t-s)H^{\beta}}(f(u(\tau))-f(v(\tau))) \, d\tau\right\|_{L^{a}} \\
   &\leq C \sum_{k=0^{\infty}} \frac{\lambda^k}{k!} t^{\sigma} \int_0^t (t-\tau)^{-\frac d{2\beta}(\frac 1r-\frac 1a)} \|(u-v)\|_{L^a}\|(|u|^{pk+m-1}+|v|^{pk+m-1})\|_{L^q} \, d\tau\\
   &\leq C \sum_{k=0^{\infty}} \frac{\lambda^k}{k!} t^{\sigma} \int_0^t (t-\tau)^{-\frac d{2\beta}(\frac 1r-\frac 1a)} \|(u-v)\|_{L^a}[\|u\|_{L^{q(pk+m-1)}}^{pk+m-1}+\|v\|_{L^{q(pk+m-1)}}^{pk+m-1}] \, d\tau.
\end{align*}  
Arguing as in \cite{Mohamed2021, fino2020cpaa}, we get for small $M$,
\[ t^{\sigma} \left\|\int_0^t e^{-(t-\tau)H^{\beta}}(f(u)-f(v)) \, d\tau\right\|_{L^{a}} \leq C M^{m-1} \, d(u,v).\]
This together with \eqref{E4.23} and \eqref{E4.6} concludes the proof of global existence for dimensions $d < 2 \beta p/(p - 1)$.
\subsection{The case $d=\frac{2\beta p}{p-1}$}
Let $u,v \in Y_M$. By using \eqref{E4.3} and Proposition \ref{lpexp}, we get
\begin{align*}
&t^{\sigma} \|\Phi(u)(t)-\Phi(v)(t)\|_{L^a}\\
\leq & t^{\sigma} \int_0^t \left\| e^{-(t-\tau)H^{\beta}}(f(u)-f(v))\right\|_{L^{a}}  \, d\tau \\
\leq & t^{\sigma} \int_0^t (t-\tau)^{-\frac d{2\beta}(\frac 1r- \frac 1a)}\|(f(u(\tau))-f(v(\tau)))\|_{L^{r}}  \, d\tau \\
  \leq & C \sum_{k=0}^{\infty} \frac{\lambda^k}{k!} t^{\sigma} \int_0^t (t-\tau)^{-\frac d{2\beta}(\frac 1r-\frac 1a)} \|(u-v)(|u|^{pk+m-1})+|v|^{pk+m-1}\|_{L^r} \, d\tau,
\end{align*}
where $1\leq r\leq a.$
Applying H\"{o}lder's inequality, we infer
\begin{align*}
&t^{\sigma} \|\Phi(u)(t)-\Phi(v)(t)\|_{L^a}\\
   &\leq C \sum_{k=0}^{\infty} \frac{\lambda^k}{k!} t^{\sigma} \int_0^t (t-\tau)^{-\frac d{2\beta}(\frac 1r-\frac 1a)} \|(u-v)\|_{L^a}\|(|u|^{pk+m-1})+|v|^{pk+m-1}\|_{L^q} \, d\tau\\
   &\leq C \sum_{k=0}^{\infty} \frac{\lambda^k}{k!} t^{\sigma} \int_0^t (t-\tau)^{-\frac d{2\beta}(\frac 1r-\frac 1a)} \|(u-v)\|_{L^a}\|u\|_{L^{q(pk+m-1)}}^{pk+m-1})+\|v\|_{L^{q(pk+m-1)}}^{pk+m-1} \, d\tau.
\end{align*}
Similar calculation yields,
\begin{align} \label{E4.14}
    t^{\sigma} \|\Phi(u)(t)-\Phi(v)(t)\|_{L^a} \leq C M^{m-1} \, \sup_{\tau>0}(\tau^{\sigma}\|(u-v)\|_{L^a})=C M^{m-1} \, \rho(u,v).
\end{align}
Here, we rely on the crucial fact that $a$ satisfies $\frac{d}{2\beta} (m-1) < a < \frac{d}{2\beta} (m-1) \frac{1}{(2-m)_+}$. Now we estimate $\|\Phi(u)(t)\|_{L^{\infty}(0,\infty; \exp L^p)}$. 
By \eqref{E4.5} and \eqref{E3.2},
\begin{align*}
&\|\Phi(u)\|_{L^{\infty}(0,\infty; \exp L^p)} 
\leq \|e^{-tH^{\beta}}u_0\|_{L^{\infty}(0,\infty; \exp L^p)}+\left\|\int_0^t e^{-(t-\tau)H^\beta}\left(f(u(\tau))\right) \, d\tau\right\|_{L^{\infty}(0,\infty; \exp L^p)}\\
&\leq {C}\|u_0\|_{\exp L^p}+ \left\|\int_0^t \, e^{-(t-\tau)H^\beta}\left(f(u(\tau))\right) \, d\tau\right\|_{L^{\infty}(0,\infty;  L^{\phi})} +\left\|\int_0^t \, e^{-(t-\tau)H^\beta}\left(f(u(\tau))\right) \, d\tau\right\|_{L^{\infty}(0,\infty;  L^p)}.
\end{align*}
By using Corollary \ref{C3.2}, we obtain 
\begin{align}
 \left\|\int_0^t \, e^{-(t-\tau)H^\beta}\left(f(u(\tau))\right) \, d\tau\right\|_{L^{\infty}(0,\infty;  L^{\phi})} \leq C M^m.   
\end{align}
Owing to \eqref{E4.3} and Proposition \ref{lpexp}, we get
\[\left\|\int_0^t \, e^{-(t-\tau)H^\beta}\left(f(u(\tau))\right) \, d\tau\right\|_{L^p}\leq C \int_0^t\|f(u(\tau))\|_{L^p} \, d\tau.\]
Likewise, we obtain
\begin{align}
    \left\|\int_0^t \, e^{-(t-\tau)H^\beta}\left(f(u(\tau))\right) \, d\tau\right\|_{L^{\infty}(0,\infty;  L^{p})} \leq C M^m.
\end{align}
Therefore, we obtain
\[\|\Phi(u)\|_{L^{\infty}(0,\infty; \exp L^p)} 
\leq {C}\|u_0\|_{\exp L^p}+2C M^m.\]
By utilizing \eqref{E4.6}, we obtain the following estimate
\[t^{\sigma} \, \|\Phi(u)\|_{L^p} 
\leq \|u_0\|_{\exp L^p}+C M^m.\]
By selecting sufficiently small values for $M$ and $\epsilon$, we ensure that $\Phi$ maps $Y_M$ to itself. {Moreover, using the inequality \eqref{E4.14}, we can show that \(\Phi\) is a contraction mapping on \(Y_M\). The desired result then follows immediately from Theorem \ref{BFP}.}
\subsection{Proof of Theorem \ref{Th1}}
Let $q\geq \max (\frac d{2\beta}, p)$. By Proposition \ref{lpexp}, we write
\begin{align*}
&\|u(t)-e^{-tH^{\beta}} u_0\|_{\exp L^p} 
\leq \int_0^t\left\| e^{-(t-\tau)H^\beta}\left(f(u(\tau))\right)\right\|_{\exp L^p} \, d\tau\\
&\leq C\int_0^t\left\| e^{-(t-\tau)H^\beta}\left(f(u(\tau))\right)\right\|_{ L^p} \, d\tau + C\int_0^t\left\| e^{-(t-\tau)H^\beta}\left(f(u(\tau))\right)\right\|_{L^{\infty}} \, d\tau\\
&\leq C\int_0^t\left\|\left(f(u(\tau))\right)\right\|_{ L^p} \, d\tau + C\int_0^t(t-\tau)^{-\frac d{2\beta q}}\left\|\left(f(u(\tau))\right)\right\|_{L^q} \, d\tau.
\end{align*}
It can be readily observed that for $r=p,\, q$, we have 
\[\left\|f(u(\tau))\right\|_{L^r}\leq C \|u\|_{\exp L^p}^m.\]
Finally, 
\begin{align*}
\|u(t)-e^{-tH^{\beta}} u_0\|_{\exp L^p} 
&\leq C\int_0^t\left(C \|u(\tau)\|_{\exp L^p}^m+(t-\tau)^{-\frac d{2\beta q}} \, C \|u(\tau)\|_{\exp L^p}^m \right) d\tau\\
&\leq C t  \|u\|_{L^{\infty}(0,\infty:\exp L^p)}^m+C t^{1-\frac d{2\beta q}}  \|u\|_{L^{\infty}(0,\infty:\exp L^p)}^m\\
&\leq C_1 t+C_2 t^{1-\frac d{2\beta q}},
\end{align*}
where $C_1,~C_2>0$ are constants. Consequently, we can conclude that $\displaystyle\lim_{t\to 0} \| u(t)- e^{-tH^{\beta}} u_0\|_{\exp L^p}=0$. Additionally, it is worth noting that the convergence $u(t) \to u_0$ as $t \to 0$ in the weak$^{\star}$ topology has been established in the analysis presented in \cite{Ioku2011}. This completes the proof of Theorem \ref{Th1}.
\end{proof}
\section{Proof of Theorem \ref{Thne}}
\label{S5}
To begin, we construct an initial data set with diverging integrability properties.
\begin{Lemma}
\label{u-0}
Let $\alpha>0$, $p>1$, and
\begin{equation} \label{defuo}
u_0(x):=\left\{\begin{array}{ll} \alpha \left(-\log |x|\right)^{\frac 1p} \quad &\text{if} \quad |x|<1,\\
0 \quad &\text{if} \quad |x|\geq1.
\end{array}\right.
\end{equation}

Then, for every $\lambda>0,$ there exists some $\tilde{\alpha}>0$ such that 
\[\int_0^{\epsilon} \int_{B_r(0)} \exp(\lambda (e^{-tH}u_0)^p)dx \, dt=\infty,\]
for every $\alpha>\tilde{\alpha},~\epsilon>0,$ and $r>0,$ where $B_r(0) \subset {\R^d}$ is the ball centered at the origin with radius $r>0.$
\end{Lemma}
\begin{proof}
Let us recall that the Weyl symbol of the Hermite semigroup $e^{-tH}$ is given by
\begin{eqnarray}\label{mc}
  e^{-tH}f(x)=C_d(\sinh(2t))^{-\frac{d}{2}} \, e^{-\frac{\tanh t}{2}|x|^2}\left(e^{-\frac{1}{2\sinh 2t}|\cdot|^2} \ast g\right)(x)  
\end{eqnarray}
where
$g(x)=e^{-\frac{\tanh t}{2}|x|^2} \, f(x)$, see \cite[Eq. (3.3)]{bhimani2022heat}.

 Fix $1>\epsilon>0, r>0$. Let $\rho=\min\{r,\frac 14\}$. Then $B_{|x|}(3x) \subset B_1(0)$ for every $|x| < \rho.$ Therefore, for any $|x| < \rho$, we have
\begin{align*}
   e^{-tH}u_0(x)&= C_d(\sinh(2t))^{-\frac{d}{2}} \, e^{-\frac{\tanh t}{2}|x|^2} \int_{|y|<1}\left(e^{-\frac{1}{2\sinh 2t}|x-y|^2} \,  e^{-\frac{\tanh t}{2}|\cdot|^2} u_0(y) \, \right)dy\\
   &\geq C_d \, \alpha \, (\sinh(2t))^{-\frac{d}{2}} \, e^{-\frac{\tanh t}{2}|x|^2} \int_{B_{|x|}(3x)}\left(e^{-\frac{1}{2\sinh 2t}|x-y|^2} \,  e^{-\frac{\tanh t}{2}|y|^2} (-\log |y|)^{\frac 1p} \, \right)dy.
\end{align*}
For $y \in B_{|x|}(3x)$, we have $2|x|\leq |y|\leq 4|x|$ and $|x|\leq |x-y|\leq 3|x|$ and thus   
\begin{align} \label{E6.1}
   e^{-tH}u_0(x)
   &\geq C_d \, \alpha \, (\sinh(2t))^{-\frac{d}{2}} \, e^{-\frac{\tanh t}{2}|x|^2} \int_{B_{|x|}(3x)}\left(e^{-\frac{1}{2\sinh 2t}|x-y|^2} \,  e^{-\frac{\tanh t}{2}|y|^2} (-\log |y|)^{\frac 1p} \, \right)dy \notag\\
   &\geq C\, C_d \, \alpha \, \left(\frac{|x|^2}{\sinh(2t)}\right)^{\frac{d}{2}} \, e^{-\frac{\tanh t}{2}|x|^2} \, e^{-\frac{9}{2\sinh 2t}|x|^2} \,  e^{-(8\tanh t)|x|^2} (-\log (4|x|))^{\frac 1p}\notag\\
   & \geq C\,  \alpha \, \left(\frac{|x|^2}{t}\right)^{\frac{d}{2}} \, \, e^{-\frac{9}{4t}|x|^2} \,  (-\log (4|x|))^{\frac 1p}.
\end{align}
 Let $\tilde{\epsilon}=\min \{\epsilon, \rho^2\}.$ Then for any $0<t<\tilde{\epsilon}$, we have $B_{\sqrt{t}}(0) \subset B_{\rho}(0).$ Hence,
\begin{align} \label{E6.31}
   \int_0^{\epsilon} \int_{|x|<r} \exp (\lambda (e^{-tH}u_0)^p) \, dx dt
   &\geq \int_0^{\tilde{\epsilon}} \int_{\frac{\sqrt{t}}{2} <|x|<\sqrt{t}} \exp (\lambda (e^{-tH}u_0)^p) dx dt\notag\\
   & \geq \int_0^{\tilde{\epsilon}} \int_{\frac{\sqrt{t}}{2} <|x|<\sqrt{t}} \exp(-\lambda {C} \alpha^p \log (4|x|)) \, dx dt \notag\\
   &\geq {C_{\alpha,\lambda} \int_0^{\tilde{\epsilon}} t^{\frac{d}{2}-\frac{\lambda C \alpha^p}{2}} dt}=\infty,
\end{align}
for $\alpha\geq \alpha_0:=\left(\frac{d+2}{C \lambda}\right)^{\frac 1p}.$ This completes the proof of Lemma \ref{u-0}.
\end{proof}
\begin{proof}[Proof of Theorem \ref{Thne}]
Let us begin by noting that the function $u_0$ defined in \eqref{defuo} belongs to $\exp L^p(\mathbb{R}^d)$ for any $\alpha > 0$. To establish the desired result, we will proceed by assuming the contrary. Specifically, suppose that there exists $T > 0$ and a nonnegative classical solution $u \in C([0, T]; \exp L^p(\mathbb{R}^d))$ to the equation \eqref{hexp}. For any $t>0,~\tau>0,~t+\tau<T$, we have 
\[u(t+\tau)=e^{-(t+\tau)H} u_0+\int_0^{t+\tau} e^{-(t+\tau-s)H} \, f(u(s)) \, ds\geq e^{-tH}u(\tau),\]
since $u\in \exp L^p(\R^d)$ is a nonnegative classical solution to \eqref{hexp}.

Next we shall show that $u(t) \geq e^{-tH}u_0\geq 0.$ To prove that we first see that as $\tau \to 0,$ we have $u(t+\tau) \to u(t)$. Now, since 
\[e^{-\frac{|x-\cdot|^2}{2 \sinh 2t}} \in L^1({\R^d}) \cap L^{\infty}({\R^d}) \subset L^1(\log L)^{\frac 12}({\R^d}),\]
for every $x\in {\R^d}$ and $u(s)$ converges in weak$^{\star}$-topology to $u_0,$ we obtain that 
\[e^{-tH}u(s,x)= C_d(\sinh(2t))^{-\frac{d}{2}} \, e^{-\frac{\tanh t}{2}|x|^2} \int_{\R^d}\left(e^{-\frac{1}{2\sinh 2t}|x-y|^2} \,  e^{-\frac{\tanh t}{2}|\cdot|^2} u(s,y) \, \right)dy,\]
converges to
\[C_d(\sinh(2t))^{-\frac{d}{2}} \, e^{-\frac{\tanh t}{2}|x|^2} \int_{\R^d}\left(e^{-\frac{1}{2\sinh 2t}|x-y|^2} \,  e^{-\frac{\tanh t}{2}|\cdot|^2} u_0(y) \, \right)dy,\]
as $s\to 0.$ Since the initial data $u_0$ is nonnegative,  we obtain
\begin{align} \label{inq6.2}
  u(t) \geq e^{-tH}u_0\geq 0.  
\end{align}
Let us choose $\phi \in C_c^{\infty}(\R^d),~\phi\geq 0$ on $\R^d$ and $\phi\geq 1$ on $B_r(0).$ Since $u$ is a nonnegative classical solution to \eqref{hexp}, we obtain
\[\frac{d}{dt}\int_{\R^d} u\phi dx+\int_{\R^d} u(-H \phi) \, dx=\int_{\R^d} f(u) \, \phi \, dx\geq \int_{B_r(0)} f(u) \, dx.\]
Therefore integrating over $\tau\in [\sigma,T'],~0<\sigma <T'<T$, we obtain
\[\int_{\R^d} u(T')\phi dx-\int_{\R^d} u(\sigma)\phi dx+\int_{\sigma}^{T'} \int_{\R^d} u(-H \phi) \, dx d\tau\geq \int_{\sigma}^{T'} \int_{B_r(0)} f(u(\tau)) \, dx d\tau.\]
Since $u\in L^{\infty}(0,T';\exp L^2(\R^d))$ and ${u}(t) \to u_0$ in weak$^{\star}$ topology, by letting $\sigma \to 0$, we obtain
\[\int_{\R^d} u(T')\phi dx-\int_{\R^d} u_0\phi dx+\int_{0}^{T'} \int_{\R^d} u(-H \phi) \, dx d\tau\geq \int_{0}^{T'} \int_{B_r(0)} f(u(\tau)) \, dx d\tau.\]
Hence 
\begin{align} \label{E6.4}
    \int_0^{T'} \int_{B_r(0)}f(u(\tau)) \, dx d\tau <\infty.
\end{align}
Now, thanks to assumption \eqref{liminf}, there are some positive constants $C>0$ and $\eta_0>0$ such that
\begin{align}\label{E6.5}
    f(\eta) \geq C\,e^{\lambda \eta^p}, \;\;\;\forall\;\; \eta>\eta_0.
\end{align}
Let us choose  $\rho<r$ and $\tilde{\epsilon}<T'$ as in the proof of  Lemma \ref{u-0}. Referring back to the proof of the previous Lemma and utilizing \eqref{E6.1}, we can deduce the following result  
\begin{align}
    e^{-tH}u_0(x)\geq C\left(\log \frac{1}{4\sqrt{t}}\right)^{\frac 1p}\geq \eta_0
\end{align} \label{E6.6}
if $t>0$ is small enough and $x\in B_{\sqrt{t}}(0)\setminus B_{\sqrt{t}/2}(0)\subset B_{\rho}(0).$

Finally by \eqref{inq6.2}, \eqref{E6.5} and \eqref{E6.6}, we obtain that
\[ \int_0^{T'} \int_{B_r(0)}f(u(\tau)) \, dx d\tau  \geq C\,\int_0^{\tilde{\epsilon}} \int_{\sqrt{t}/2\leq |x| \leq \sqrt{t}}\exp\left(\lambda (e^{-tH}u_0)^p\right) \, dx d\tau ,\]
which contradicts \eqref{E6.4} and \eqref{E6.31} if $\alpha>0$ is sufficiently large. This finishes the proof of Theorem \ref{Thne}.
\end{proof}
\section{{Conclusion and Open Problems}}
\label{Conc}
Evolution partial differential equations with exponential nonlinearities arise naturally in a wide range of real-world applications, where the growth or decay of physical, biological, or chemical quantities is governed by nonlinear mechanisms. For detailed discussions, see, among others, \cite{Boyd, Biology1, Lam77, Biology, Williams}.

In this paper, we have investigated the Cauchy problem for a heat equation involving a fractional harmonic oscillator and an exponential nonlinearity. We established local well-posedness in appropriate Orlicz spaces and obtained global existence results for small initial data in these spaces. Additionally, we provided precise decay estimates for large time, highlighting the influence of the nonlinearity's behavior near the origin. Furthermore, we demonstrated that the existence of local nonnegative classical solutions is not guaranteed for certain nonnegative initial data in the appropriate Orlicz space.

Our results generalize earlier studies on heat equations with polynomial nonlinearities \cite{bhimani2022heat} to the case of exponential nonlinearities, which play a pivotal role in numerous physical models, particularly those describing self-trapped beams in plasma. By employing Orlicz spaces, we were able to effectively address the challenges introduced by the exponential terms and derive results that parallel those of the classical heat equation \cite{majdoub2018ade, majdoub2018picm, Mohamed2021}. This approach not only broadens the scope of existing theory but also provides a robust framework for analyzing exponential nonlinearities in similar contexts.

Despite these advancements, several open problems remain. The range of parameters for which the decay estimates hold could potentially be improved, and further investigation is needed to determine the optimal conditions under which these estimates are valid. While we focused on exponential nonlinearities, it would be interesting to explore the behavior of solutions for other types of nonlinearities, such as those with different growth rates or more complex structures. Extending our results to more general operators, including those with variable coefficients or broader classes of potentials, could provide deeper insights into the behavior of solutions in more complex and physically realistic settings.

Understanding the conditions under which solutions blow up in finite time, especially for large initial data, is an important direction for future research. This could involve studying the interplay between the nonlinearity and the fractional harmonic oscillator.

\end{document}